\documentclass[10pt]{amsart}

\usepackage{amssymb, amsmath, amsfonts, amsthm}
\usepackage{amsaddr}
\usepackage{color}

\newtheorem{defn}[equation]{Definition}
\newtheorem{lemma}[equation]{Lemma}

\newtheorem{thm}[equation]{Theorem}

\newcommand{\F}{F_4(2)}
\newcommand{\E}{E_6(2)}

\newcommand{\OO}{O_{10}^+(2)}
\newcommand{\LL}{L_6(2)}
\newcommand{\DD}{{}^3D_4(2)}
\newcommand{\ns}{\dot{\;}}
\newcommand{\sg}{\leqslant}
\newcommand{\sdp}{\colon\!}

\title{The Hurwitz subgroups of $E_6(2)$}
\author{Emilio Pierro$^\dag$}
\address{Fakult\"at f\"ur Mathematik, \\ Universit\"at Bielefeld, \\ 33602, Bielefeld, \\ Germany}

\curraddr{Centre for the Mathematics of Symmetry and Computation\\
School of Mathematics and Statistics\\
The University of Western Australia\\
35 Stirling Highway, Crawley, WA 6009, Australia.}
\email{emilio.pierro@uwa.edu.au}
\date{\today}

\begin{document}
\maketitle

\begin{abstract} We prove that the exceptional group $\E$ is not a Hurwitz group. In the course of proving this, we complete the classification up to conjugacy of all Hurwitz subgroups of $\E$, in particular, those isomorphic to $L_2(8)$ and $L_3(2)$.\end{abstract}

\section{Introduction}
Hurwitz groups are finite quotients of the triangle group
\smallskip
\[\Delta := \Delta(2,3,7) = \langle x, y, z \; | \; x^2 = y^3 = z^7 = xyz =1 \rangle\]
and interest in them stems from their connection to Riemann surfaces. A result due to Hurwitz \cite{hurwitz} states that if $\Sigma_g$ is a compact Riemann surface of genus $g > 1$ and $G$ is a group of conformal automorphisms of $\Sigma_g$, then the order of $G$ is bounded above by $84(g-1)$. Groups which attain this bound are known as Hurwitz groups and for a finite group $G$ to be a Hurwitz group, this is equivalent to $G$ being a finite quotient of $\Delta(2,3,7)$. We refer the reader to the most recent and most excellent survey on Hurwitz groups by Conder \cite{cond2} and the references therein for more information.

\smallskip 

Rather unexpectedly, the study of Hurwitz groups becomes extremely useful for capturing and classifying conjugacy classes of small non-abelian finite quasi-simple subgroups of a finite group $G$. The purpose of this paper is then twofold. Our main result is the following.

\begin{thm} \label{thm} Let $H \sg \E$ be a Hurwitz group. Then $H$ is isomorphic to $L_2(8)$, $L_3(2)$, $2^3 \ns L_3(2)$ or $\DD$. In particular, $\E$ is not a Hurwitz group.\end{thm}

In the course of proving the above theorem it becomes necessary to classify, up to conjugacy, subgroups of $\E$ of the isomorphism types named in Theorem \ref{thm}. These are then listed in Table \ref{e62table}. The majority of the work towards the proof of the above theorem and Table \ref{e62table} is due to Kleidman and Wilson \cite{maxe62} which we shall cite almost continuously. We note, however, that there is a slight error \cite[pg.288]{maxe62} where the authors identify $\F$ as a Hurwitz group. This is known to be false (see Table \ref{f42table}) and is corrected by Lemma \ref{ccd}. Following the suggestion of Wilson to the author, the proof of Theorem \ref{thm} does not rely on the main result of \cite{maxe62}. We are able to use results from \cite{maxe62} to prove almost everything, however for certain conjugacy classes of subgroups we rely on GAP \cite{gap} to determine the precise structure of their normalisers.

\smallskip 

In and of itself, Theorem \ref{thm} is a minor result in the study of Hurwitz groups. However, it is a necessary first step towards determining which members of the families of exceptional groups of Lie type of (twisted) rank at least $4$ are Hurwitz groups. For exceptional groups of (twisted) rank less than $4$, we summarise the known results in Table \ref{exchurw}. We also remark that for the groups $^2F_4(2^{2m+1})$, the condition appearing there is equivalent to the order of the group being divisible by $7$. In addition to those mentioned, the group $F_4(2)$ is known not to be a Hurwitz group, see for example Table \ref{f42table}. For general $q$, however, it is not known whether $F_4(q)$ is a Hurwitz group.

\begin{table}[h]\centering
\begin{tabular}{c | c c c}
$K$				&Hurwitz				&Reference\\ \hline
$^2B_2(2^{2m+1})$	&None				&--\\
$^2G_2(3^{2m+1})$	&$m \geq 1$			&Malle \cite{malle}, Jones \cite{rr}\\
$^2F_4(2^{2m+1})$	&$m \equiv 1$ mod $3$	&Malle \cite{mallesmallhurwitz}\\
$^3D_4(p^n)$		&$p \neq 3$, $p^n \neq 4$&Malle \cite{mallesmallhurwitz}\\
$G_2(p^n)$		&$p^n \geq 5$			&Malle \cite{malle}\\
\end{tabular}
\caption{Status of exceptional groups of Lie type as Hurwitz groups} \label{exchurw}
\end{table}

Our method, which is standard, is outlined in some detail in Section \ref{method}. In brief, given a finite group $G$, it is possible to determine the order of Hom$(\Delta,G)$ and then using knowledge of the subgroup structure of $G$, determine which elements of this set are in fact epimorphisms. In Section \ref{aux} we record various necessary preliminary results about the structure of $\E$ and in Section \ref{sc} we account for the elements of Hom$(\Delta,G)$. In the Appendix we record the conjugacy classes and normalisers of Hurwitz subgroups of the groups $\LL$, $\OO$ and $\F$ which will be necessary for the proof of Theorem \ref{thm}. We provide these without proof since they can be easily determined in GAP.

\smallskip

We use ATLAS \cite{atlas} notation throughout. In particular, we use the notation $nX$ to denote a conjugacy class of elements of order $n$, ordered alphabetically in decreasing size of their centraliser. If $nX$ and $nY$ are two conjugacy classes where elements from $nX$ are powers of elements of $nY$, then we write the union of these two classes as $nXY$. When we consider the normaliser of a subgroup $H$ in $\E$, we shall normally write this as $N_E(H)$ for brevity. A similar practice will be used for the subgroups isomorphic to $\F$, $\OO$, $\LL$ and $\DD$ with the obvious abbreviations.

\subsection*{Acknowledgements.} 
The author is grateful to Rob Wilson and Alastair Litterick for many helpful conversations in the preparation of this paper and to Marston Conder for assistance in the verification of these results. The author also thanks the referee for their thorough reading of the manuscript.

This work was supported by the SFB 701 and by ARC grant DP140100416.

\section{Preliminaries}\label{method}
The techniques we employ are standard and `well-known' to those in the area. We include a description of them here so as to be as self-contained as possible. We begin with some terminology.

\begin{defn} Let $G$ be a finite group and let $2X$, $3Y$ and $7Z$ be $G$-conjugacy classes of elements. If there exists $x \in 2X$, $y \in 3Y$ and $z \in 7Z$ such that $xyz=1$, then we call $(x,y,z)$ a Hurwitz triple and $H = \langle x,y,z \rangle$ a Hurwitz subgroup of type $(2X,3Y,7Z)$, or usually a $(2X,3Y,7Z)$-subgroup.\end{defn}

In order to count the number of elements of the set
\[\mathcal{H} = \{(x, y, z) \in G \times G \times G \mid x^2=y^3=z^7=xyz=1\}\]
it is necessary to determine the structure constants of $G$. Note that one can naturally identify elements of $\mathcal{H}$ with elements of Hom$(\Delta,G)$. For this we use the well known structure constant formula due to Frobenius \cite{frob} as follows. If $C_1$, $C_2$ and $C_3$ are three not necessarily distinct conjugacy classes in a finite group $G$, then the order of the set
\[\mathcal{C} = \{(x,y,z) \in C_1 \times C_2 \times C_3 \mid xyz=1\}\]
is given by the following formula
\[\vert \mathcal{C} \vert = \frac{\vert C_1 \vert \vert C_2 \vert \vert C_3 \vert}{\vert G \vert} \sum_{\chi \in \mathrm{Irr}(G)}\frac{\chi(x)\chi(y)\chi(z)}{\chi(1)}.\]
Note that we give the specific form for three conjugacy classes but such a formula exists for any finite number of conjugacy classes. In practice, we shall `normalise' this constant and define
\[n_G(C_1,C_2,C_3)=\frac{\vert \mathcal{C} \vert}{\vert G \vert}.\]
It will also be necessary to determine the structure constants within Hurwitz subgroups of $G$. If $H$ is a Hurwitz subgroup of $G$ of type $(2X,3Y,7Z)$, where $2X$, $3Y$ and $7Z$ are $G$-conjugacy classes, then, by abuse of notation, we denote by $n_H(2X,3Y,7Z)$ the sum over all combinations of normalised structure constants $n_H(C_i,C_j,C_k)$ where $C_i$, $C_j$ and $C_k$ are $H$-conjugacy classes of elements which fuse to the $G$-classes $2X$, $3Y$ and $7Z$ respectively. Since ultimately we shall only deal with four isomorphism classes of Hurwitz subgroups of $G$, this should not cause too much confusion.

Given a Hurwitz triple $(x,y,z)$ in $G$ of type $(2X,3Y,7Z)$ this clearly generates a unique Hurwitz subgroup $H = \langle x,y \rangle \sg G$. Letting $H_i$ for $i \in I$ denote the $G$-conjugacy class representatives of $(2X,3Y,7Z)$-subgroups, we obtain the following formula
\[n_G(2X,3Y,7Z) = \sum_{i \; \in \; I} \frac{n_{H_i}(2X,3Y,7Z)}{[N_G(H_i):H_i]}.\]

\section{Auxiliary structural results}\label{aux}
In this section we collect a number of relevant results from \cite{maxe62} and prove further results that will be necessary for classifying normalisers in $\E$ of Hurwitz subgroups. We begin with some local analysis.

We reproduce from \cite[Table 2]{maxe62} the conjugacy classes and normalisers of elements of orders $2$, $3$ and $7$ in $\E$ for reference.
\begin{align*}
N_E(2A) &\cong 2^{1+20}_+ \sdp \LL\\
N_E(2B) &\cong [2^{24}] \sdp S_6(2) \sg 2^{16} \sdp \OO\\
N_E(2C) &\cong [2^{27}] \sdp (L_2(2) \times L_3(2)) \sg 2^{2+9+18} \sdp (L_2(2) \times L_3(2) \times L_3(2))\\
N_E(3A) &\cong L_2(2) \times \LL\\
N_E(3B) &\cong (3 \times O_8^-(2)) \sdp 2 \sg \OO\\
N_E(3C) &\cong 3 \ns (3^2 \sdp Q_8 \times L_3(4)) \sdp S_3\\
N_E(7AB) &\cong (7 \times \DD) \sdp 3\\
N_E(7C) &\cong (7 \sdp 3 \times L_3(2) \times L_3(2)) \sdp 2 \sg (L_3(2) \times L_3(2) \times L_3(2)) \sdp S_3\\
N_E(7D) &\cong (7^2 \sdp 3 \times L_3(2)) \sdp 2 \sg (L_3(2) \times L_3(2) \times L_3(2)) \sdp S_3
\end{align*}

\smallskip

The character table of $\E$ is known and can be found in GAP. The non-zero normalised $(2,3,7)$-structure constants are then as follows.
\begin{align*}
n_E(2A,3A,7C)	&= 1/28224	\quad & \quad n_E(2C,3A,7C)&=11/192\\
n_E(2B,3A,7C) &=1/288		\quad & \quad n_E(2C,3B,7D)&=43/24\\
n_E(2B,3B,7AB) &=1/6048	\quad & \quad n_E(2C,3C,7C)&=15/7\\
n_E(2B,3B,7D) &= 5/56		\quad & \quad n_E(2C,3C,7D)&=137/21
\end{align*}

\smallskip

An important class of subgroups of $\E$ are those isomorphic to $L_3(2)^3$ whose existence and uniqueness up to conjugacy follows from \cite[Propositions 2.3, 5.4]{maxe62}. For the rest of the paper we denote by $T$ the normaliser in $\E$ of such a subgroup, that is
\[T \cong (L_3(2) \times L_3(2) \times L_3(2)) \sdp S_3.\]
Note that the action of an element of order $3$ in $T \setminus T''$ permutes the factors of $T''$ cyclically, whereas an element of order $2$ acts as an outer automorphism on each factor of $T''$ and interchanges two of its factors \cite[Proposition 2.3]{maxe62}. The fusion of elements of $T$ of orders $3$ or $7$ is given in \cite[Lemma 2.6]{maxe62} and \cite[Lemma 8.3]{maxe62}. To summarise, let $t$ be an element of order $3$ or $7$ in $T''$. \begin{enumerate}
\item If $t$ projects non-trivially onto one factor of $T''$, then $t \in 3A$ or $7C$.
\item If $t$ projects non-trivially onto two factors of $T''$, then $t \in 3B$, $7AB$ or $7D$.
\item If $t$ projects non-trivially onto three factors of $T''$, then $t \in 3C$, $7C$ or $7D$.
\end{enumerate}

The following lemma determines the fusion of a subgroup of $T$ isomorphic to $L_3(2)$ using the above information.

\begin{lemma} \label{tsubgroups} Let $H$ be a subgroup of $T$ with $H \cong L_3(2)$. Then $H$ is of one of the following $\E$-types: $(2A,3A,7C)$, $(2B,3B,7AB)$, $(2B,3B,7D)$, $(2C,3C,7C)$ or $(2C,3C,7D)$.\end{lemma}

\begin{proof} Let $H \cong L_3(2) \sg T$. We may immediately reduce to the case that the structure constant corresponding to the type of $H$ is non-zero.
If $H$ projects onto one factor of $T''$, then $H$ contains subgroups isomorphic to $7C \sdp 3A$. Since neither the centralisers of $2B$- nor $2C$-involutions contain $L_3(2) \times L_3(2)$, $H$ must be of type $(2A,3A,7C)$. 
If $H$ projects onto two factors, then $H$ contains subgroups isomorphic to $7AB \sdp 3B$ or $7D \sdp 3B$ and hence $H$ is of type $(2B,3B,7AB)$, $(2B,3B,7D)$ or $(2C,3B,7D)$. However, type $(2C,3B,7D)$ cannot occur by \cite[Proposition 8.4]{maxe62} since the centraliser in $\E$ of $H$ contains a subgroup of shape $L_3(2)$.
Finally, if $H$ projects onto three factors of $H$, then $H$ contains subgroups isomorphic to $7C \sdp 3C$ or $7D \sdp 3C$, hence $H$ must be of type $(2C,3C,7C)$ or type $(2C,3C,7D)$.\end{proof}

In the sequel, we shall need to classify many conjugacy classes of subgroups of $\E$ isomorphic to $L_3(2) \cong L_2(7)$. In particular, the analysis of subgroups isomorphic to $7 \sdp 3$ will be crucial. We summarise the necessary results in the following lemma.

\begin{lemma} \label{parafusion} Let $P \cong 7 \sdp 3 \sg L_3(2) \sg \E$.\begin{enumerate}
\item If elements of order $7$ in $P$ fuse to $\E$-class $7C$, then either
\begin{enumerate}
\item $P \cong 7C \sdp 3A$ and $N_E(7C \sdp 3A) \cong (7 \sdp 3 \times L_3(2) \times L_3(2)) \sdp 2$ or
\item $P \cong 7C \sdp 3C$ and $N_E(7C \sdp 3C) \cong	3 \times ((3 \times (7 \sdp 3)) \sdp 2)$
\end{enumerate}
\item If elements of order $7$ in $P$ fuse to $\E$-class $7D$, then either
\begin{enumerate}
\item $P \cong 7D \sdp 3B$ and $N_E(7D \sdp 3B) \cong (7 \sdp 3 \times L_3(2)) \sdp 2$, or
\item $P \cong 7D \sdp 3C$ and $N_E(7D \sdp 3C) \cong 7 \sdp 6 \times 3$
\end{enumerate}
\end{enumerate}
\end{lemma}

\begin{proof} Let $P \cong 7 \sdp 3$ where $z \in P$ belongs to $7C$ or $7D$ and $y \in P$ has order $3$. Let $T'' \cong K_1 \times K_2 \times K_3$ where $K_i \cong L_3(2)$. Since $P$ is contained in a conjugate of $T''$ we see from Lemma \ref{tsubgroups} that $P$ is isomorphic to one of the four types given.
For the proof of (1), suppose that $z \in 7C$. Without loss of generality we can assume that $z$ projects non-trivially onto only one factor of $T''$, say $K_1$. If $y \in 3A$, then, since $y$ also projects non-trivially onto $K_1$, the structure of $N_E(P)$ is clear. Now suppose that $y \in 3C$ so that $y$ projects non-trivially onto all three factors of $T''$. It is clear that the only elements of $T''$ centralising $P$ are elements of order $3$ and $C_E(P)$ is normalised by elements of order $2$ but not $3$ in $T \setminus T''$. The proof of part (2) is contained in the proof of \cite[Proposition 8.4]{maxe62}, which completes the proof.\end{proof}

We now turn to large simple subgroups of $\E$. The isomorphism classes of simple subgroups of $\E$ are determined in \cite[Section 4]{maxe62}. Those which will be of most importance to us are those isomorphic to $\F$, $\OO$, $\LL$ and $\DD$. From \cite[Section 6]{maxe62} we see that subgroups of $\E$ isomorphic to $\F$ are all conjugate, as are those isomorphic to $\OO$ or $\LL$. Subgroups isomorphic to $\DD$ fall into two conjugacy classes and any $\DD$ subgroup is contained in a subgroup isomorphic to $\F$. The fusion of conjugacy classes of $\F$, $\OO$ and $\LL$ into $\E$ are given in Table \ref{fusion}. These can easily be determined by the power maps given in \cite{atlas} and by GAP. The fusion of conjugacy classes of $\DD$ into $\F$ can also be determined. This is necessary to prove the following lemma.

\begin{lemma} \label{3d4comm} Let $H \cong \DD \sg \E$. If $N_E(H) \cong (7 \times \DD) \sdp 3 \sg \E$, then a Hurwitz generating triple of $H$ is of $\E$-type $(2C,3C,7C)$.\end{lemma}

\begin{proof} Let $H$ be as in the hypothesis and let $(x,y,z)$ be a Hurwitz generating triple for $H$. By \cite[Proposition 4]{mallesmallhurwitz}, $(x,y,z)$ is of $\DD$-type $(2B,3B,7D)$. The fusion of $\DD$ into $\F$ can easily be determined (up to outer automorphism of $\F$) and from the fusion of $\F$ into $\E$ we see that $(x,y,z)$ is of $\E$-type either $(2C,3C,7C)$ or $(2C,3C,7D)$.

Let $P \cong 7 \sdp 3 \sg \DD$. From \cite[pg. 89]{atlas} we see that $P$ does not contain elements from $\DD$-class $7ABC$, since such elements are normalised only by elements of order $2$. If $P$ is contained in a maximal subgroup of $\DD$ isomorphic to $(7 \times L_2(8)) \sdp 2$, then elements of order $3$ in $P$ fuse to $\DD$-class $3A$, since $C_D(3B)$ is not divisible by $7$. Since all elements of order $3$ in $N_D(7D)$ are $\DD$-conjugate, it follows that $P \cong 7D \sdp 3A$ is unique up to conjugacy in $\DD$.

Since $N_D(3A) \cong S_3 \times L_2(8)$, these elements fuse to $\F$-class $3A$ or $3B$ and hence $\E$-class $3A$ or $3B$. Then, by Lemma \ref{parafusion}, subgroups isomorphic to $P$ are of $\E$-type $7C \sdp 3A$ or $7D \sdp 3B$. Since $P$ is contained in a conjugate of $C_D(7ABC) \cong 7 \times L_3(2)$, if $C_E(H) \cong 7$, then $N_E(P)$ contains $7 \times (7 \sdp 3) \times L_3(2)$ and hence the $\DD$-class $7D$ fuses to the $\E$-class $7C$. This completes the proof.\end{proof}

\section{The Hurwitz subgroups of $\E$}\label{sc}
We now account for all of the conjugacy classes of Hurwitz subgroups of $\E$ and prove that they are as stated in Table \ref{e62table}. The proof is organised into lemmas according to the type of each Hurwitz subgroup.

\subsection*{$(2A,3A,7C)$-subgroups}
\begin{lemma} There is a unique conjugacy class of $(2A,3A,7C)$-subgroups in $\E$. If $H$ belongs to this class, then $H \cong L_3(2)$ and $N_E(H) \cong (L_3(2) \times L_3(2) \times L_3(2)) \sdp 2$.\end{lemma}
\begin{proof} From the examination of the normalisers of $p$-elements \cite[Table 2]{maxe62}, we see that a factor of $T''$ must be of type $(2A,3A,7C)$. From the discussion of the action of elements in $T \setminus T''$, it is clear that subgroups projecting onto one factor are conjugate in $\E$ and that $N_E(H) \cong (L_3(2) \times L_3(2) \times L_3(2)) \sdp 2$. This accounts for the entire $n_E(2A,3A,7C)$ structure constant.\end{proof}

%
\subsection*{$(2B,3A,7C)$-subgroups}
\begin{lemma} There is a unique conjugacy class of $(2B,3A,7C)$-subgroups in $\E$. If $H$ belongs to this class, then $H \cong L_3(2)$ and $N_E(H) \cong L_3(2) \times S_4 \times S_4$.\end{lemma}
\begin{proof}
From Table \ref{l62table} we see that $\E$ contains a class of $(2B,3A,7C)$-subgroups with representative $H \cong L_3(2)$ and $N_L(H) \cong H \times S_4$. Let $S$ denote $C_L(H) \cong S_4$, let $y \in S$ have order $3$ and note that $y$ belongs to $\E$-class $3A$ since it belongs to $\LL$-class $3B$. Then, $C_E(y)' \cong \LL$ contains $H \times S_4$. Since $C_E(y)' \cap S$ is trivial, it follows that $N_E(H)$ contains $C \cong H \times S_4 \times S_4$. Let $P \cong 7C \sdp 3A$. Since the group $P \times S_4 \times S_4$ is self-normalising in $\E$ it follows that $N_E(H) = C$. Since this accounts for the full structure constant, the proof is complete.\end{proof}

%
\subsection*{$(2B,3B,7AB)$-subgroups}
\begin{lemma} There is a unique conjugacy class of $(2B,3B,7AB)$-subgroups in $\E$. If $H$ belongs to this class, then $H \cong L_3(2)$ and $N_E(H) \cong L_3(2) \times G_2(2)$.\end{lemma}
\begin{proof} The existence and structure follow from \cite[Proposition 5.2]{maxe62} and we see immediately that they account for the entire structure constant.\end{proof}

%
\subsection*{$(2B,3B,7D)$-subgroups}
\begin{lemma} There are two conjugacy classes of $(2B,3B,7D)$-subgroups in $\E$ with representatives $H$ as follows:
\begin{enumerate}
\item $H \cong L_3(2)$ and $N_E(H) \cong (L_3(2) \times L_3(2)) \sdp 2$; or,
\item $H \cong L_3(2)$ and $N_E(H) \cong L_3(2) \times S_4$.
\end{enumerate}
\end{lemma}
\begin{proof} The existence and structure follows from \cite[Proposition 8.4]{maxe62}. It remains to observe that these classes account for the entire structure constant.\end{proof}

%
\subsection*{$(2C,3A,7C)$-subgroups}
\begin{lemma} \label{cac} There are two conjugacy classes of $(2C,3A,7C)$-subgroups in $\E$ with representatives $H$ as follows:
\begin{enumerate}
\item $H \cong L_3(2)$ and $N_E(H) \cong (L_3(2) \times D_8 \times D_8) \sdp 2$ or
\item $H \cong 2^3 \ns L_3(2)$ and $N_E(H) \cong 2^3 \ns L_3(2) \times (2^2 \times 2^2) \sdp S_3$.
\end{enumerate}
\end{lemma}
\begin{proof}
For the proof we use GAP to explicitly construct conjugacy class representatives and their normalisers in $\E$. We take as generators for $\E$ those given in \cite[Section 3.3]{howlett} since this allows us to easily construct the parabolic and maximal rank subgroups of $\E$. From Table \ref{f42table} we see that there exists $(2C,3A,7C)$-subgroups of $\E$ isomorphic to $L_3(2)$ and $2^3 \ns L_3(2)$. Let $P \cong 7C \sdp 3A$ denote the normaliser in $H$ of a Sylow $7$-subgroup and note that $P$ does not depend on the isomorphism type of $H$. We then construct the normaliser in $\E$ of a $3A$-element in $P$ and check which of the $1113210$ $2C$-involutions inverting the $3A$ element generate with $P$ a subgroup isomorphic to $H$. Of these, there are $441$ distinct subgroups isomorphic to $L_3(2)$ and $588$ distinct subgroups isomorphic to $2^3 \ns L_3(2)$. We then see that under the action of $N_E(P)$ isomorphic groups belong to the same orbit, hence $H$ is determined up to conjugacy by its isomorphism type. We are then able to construct the full normaliser in $\E$ of each isomorphism type and they are as stated. Finally, since these conjugacy classes account for the entire structure constant, the proof is complete.\end{proof}

%
\subsection*{$(2C,3B,7D)$-subgroups}
\begin{lemma} There are four conjugacy classes of $(2C,3B,7D)$-subgroups in $\E$ with representatives $H$ as follows:
\begin{enumerate}
\item $H \cong L_3(2)$ and $N_E(H) \cong (L_3(2) \times D_8) \sdp 2$
\item $H \cong 2^3 \ns L_3(2)$ and $N_E(H) \cong 2^3 \ns L_3(2) \times S_4$
\item $H \cong 2^3 \ns L_3(2)$ and $N_E(H) \cong 2^3 \ns L_3(2) \times D_8$ or
\item $H \cong 2^3 \ns L_3(2)$ and $N_E(H) \cong 2^3 \ns L_3(2) \times 2^2$.
\end{enumerate}
\end{lemma}
\begin{proof} By \cite[Proposition 8.4]{maxe62} there exists a unique conjugacy class of subgroups $H \cong L_3(2)$ in $\E$ and $N_E(H)$ is as given. From Table \ref{f42table} we see that there exist $(2C,3B,7D)$-subgroups isomorphic to $2^3 \ns L_3(2)$ and we now determine the conjugacy classes of such subgroups. Let $H \cong 2^3 \ns L_3(2)$ and let $P \cong 7D \sdp 3B$ denote the normaliser in $H$ of a Sylow $7$-subgroup. As in the proof of the previous lemma, we use GAP to construct conjugacy class representatives and their normalisers in $\E$ for each of the remaining classes. There are $128520$ $2C$-involutions inverting a $3B$-element in $P$. Gathering these together, there are $140$ distinct subgroups isomorphic to $2^3 \ns L_3(2)$ generated by $P$ and any such involution. Under the action of $N_E(P) \cong (P \times L_3(2)) \sdp 2$ these fall into three orbits of lengths $14$, $42$ and $84$ with normalisers in $\E$ as given. Since these four conjugacy classes account for the entire structure constant, the proof is complete.\end{proof}

%
\subsection*{$(2C,3C,7C)$-subgroups}
\begin{lemma} There are four conjugacy classes of $(2C,3C,7C)$-subgroups in $\E$ with representatives $H$ as follows:
\begin{enumerate}
\item $H \cong \DD$ and $N_E(H) \cong (7 \times \DD) \sdp 3$,
\item $H \cong L_3(2)$ and $N_E(H) \cong (3 \times L_3(2)) \sdp 2$,
\item $H \cong L_2(8)$ and $N_E(H) \cong S_3 \times (L_2(8) \sdp 3)$, or,
\item $H \cong L_2(8)$ and $N_E(H) \cong 2 \times L_2(8)$.
\end{enumerate}
\end{lemma}
\begin{proof}
From \cite[Proposition 6.11]{maxe62} we see that subgroups isomorphic to $\DD$ exist and fall into two conjugacy classes distinguished by their normaliser in $\E$. It follows from Lemma \ref{3d4comm} that if $H \cong \DD$, then $N_E(H) \cong (7 \times \DD) \sdp 3$.

From Table \ref{o10table} we see that subgroups isomorphic to $\OO$ contain a unique conjugacy class of subgroups isomorphic to $L_3(2)$ of $\E$-type $(2C,3C,7C)$. Let $H \cong L_3(2)$ denote a representative from this class. Since $H$ is centralised in $\OO$ by an element of order $3$ which fuses to an element from the $\E$-class $3B$, we see that the normaliser in $\E$ of $H$ is equal to $N_O(H) \cong (3 \times L_3(2)) \sdp 2$.

Now suppose that $H \cong L_2(8)$ and note that there exist subgroups of $\OO$ satisfying the hypotheses. Consider a dihedral subgroup of order $18$ in $H$. Since elements of order $9$ in $\OO$ belong to $\E$-class $9A$ and $N_E(9A) \cong 9 \sdp 6 \times S_3$, it follows that $N_E(H) \sg S_3 \times L_2(8) \sdp 3$. By also considering the centraliser in $\E$ of a dihedral group of shape $7C \sdp 3C$, we see that non-trivial elements of $C_E(H)$ belong to $\E$-classes $2B$ or $3B$. We see from Table \ref{o10table}, that there exists a conjugacy class of subgroups $H \cong L_2(8)$ satisfying the hypothesis with $N_E(H) \cong S_3 \times H \sdp 3$, and all such $H$ which centralise an element of order $3$ are conjugate in $\E$. Finally, we are able to find check in GAP that the conjugacy class of subgroups $H$ in $\OO$ with $N_O(H) \cong 2 \times H$ are unique up to conjugacy in $\E$ and $N_E(H) = N_O(H)$. Since these four conjugacy classes account for the entire structure constant, the proof is complete.\end{proof}

%
\subsection*{$(2C,3C,7D)$-subgroups}
\begin{lemma} \label{ccd} There are $66$ conjugacy classes of $(2C,3C,7D)$-subgroups in $\E$ with representatives $H$ as follows:
\begin{enumerate}
\item $H \cong \DD$ and $N_E(H) \cong \DD \sdp 3$,
\item $H \cong L_3(2)$ and $N_E(H) \cong L_3(2) \sdp 2$,
\item $H \cong L_2(8)$ and $N_E(H) \cong S_3 \times (7 \times L_2(8)) \sdp 3$; or,
\item $H \cong L_2(8)$ and $N_E(H) \cong 2 \times (7 \times L_2(8)) \sdp 3$ of which there are $63$ classes.
\end{enumerate}
\end{lemma}
\begin{proof}
From Lemma \ref{3d4comm} and \cite[Proposition 6.11]{maxe62} subgroups isomorphic to $\DD$ exist and are accounted for up to conjugacy. If $H \cong L_2(8)$, then $H$ is determined up to conjugacy in \cite[Section 9]{maxe62}. Now suppose that $H \cong L_3(2)$ and note that there exists a conjugacy class of subgroups in $\F$ satisfying the hypothesis with $N_F(H) \cong L_3(2) \sdp 2$. By Lemma \ref{parafusion} if the centraliser $C_E(H)$ is non-trivial, then it has order $3$. But, from Tables \ref{l62table} and \ref{o10table} and from the structure of $C_E(3C)$ we see that $C_E(H)$ does not contain elements of order $3$. Hence $N_E(H) \cong L_3(2) \sdp 2$. Since the entire structure constant has now been accounted for, the proof is complete.\end{proof}

\bibliographystyle{amsplain}


\appendix\section{Hurwitz subgroups of $\E$, $\F$, $\OO$ and $\LL$}
\begin{table}[h]\centering
\begin{tabular}{c c c c c c c c c c c c }
Type			&$H$			&$N_E(H)$							&Contribution	&Overgroups of $H$\\
\hline
$(2A,3A,7C)$	&$L_3(2)$			&$(H \times L_3(2) \times L_3(2)) \sdp 2$		&1/28224		&$T$\\
\hline
$(2B,3A,7C)$	&$L_3(2)$			&$H \times S_4 \times S_4$				&1/288		&$\F$, $\OO$, $\LL$\\
\hline
$(2B,3B,7AB)$	&$L_3(2)$			&$H \times G_2(2)$						&1/6048		&$\LL$, $T$\\
\hline
$(2B,3B,7D)$	&$L_3(2)$			&$(H \times L_3(2)) \sdp 2$				&1/168		&$\LL$\\
			&$L_3(2)$			&$H \times S_4$						&1/12		&$\F$\\
\hline
$(2C,3A,7C)$	&$L_3(2)$			&$(H \times D_8 \times D_8) \sdp 2$			&1/64		&$\F$, $\OO$\\
			&$2^3 \ns L_3(2)$	&$H \times (2^2 \times 2^2) \sdp S_3$		&1/24		&$\F$, $\OO$\\
\hline
$(2C,3B,7D)$	&$L_3(2)$			&$(H \times D_8) \sdp 2$					&1/8			&$\F$\\
			&$2^3 \ns L_3(2)$	&$H \times D_8$						&1/2			&$\F$\\
			&$2^3 \ns L_3(2)$	&$H \times S_4$						&1/6			&$\F$\\
			&$2^3 \ns L_3(2)$	&$H \times 2^2$						&1			&$\LL$\\
\hline
$(2C,3C,7C)$	&$L_3(2)$			&$(3 \times H) \sdp 2$					&1/3			&$N_E(3B)$\\
			&$L_2(8)$			&$H \times  2$							&3/2			&$N_E(2B)$\\
			&$L_2(8)$			&$S_3 \times (H \sdp 3)$					&1/6			&$N_E(2B)$, $N_E(3B)$\\
			&$^3D_4(2)$		&$(7 \times H) \sdp 3$					&1/7			&$N_E(7A)$\\
\hline
$(2C,3C,7D)$	&$L_3(2)$			&$H \sdp 2$							&1			&$\F$\\
			&$L_2(8)$			&$S_3 \times (7 \times H) \sdp 3$			&1/42		&$\F$, $\LL$\\
			&$L_2(8)^*$		&$2 \times (7 \times H) \sdp 3$				&1/14		&$N_E(2A)$\\
			&$^3D_4(2)$		&$H \sdp 3$							&1			&$\F$\\
			\hline
\end{tabular}
\caption{Conjugacy classes of Hurwitz subgroups of $\E$. \newline In the case of (*) there are $63$ classes.}
\label{e62table}
\end{table}

\begin{table}[h]\centering
\begin{tabular}{c | c c c c c c c c c c c }
$\E$		&$2A$	&$2B$	&$2C$	&$3A$	&$3B$	&$3C$	&$7AB$	&$7C$	&$7D$	&$9A$	&$9B$\\ \hline
$L_6(2)$	&$2A$	&$2B$	&$2C$	&$3B$	&$3C$	&$3A$	&$7AB$	&$7CD$	&$7E$	&--		&$9A$\\
$O_{10}^+(2)$	&$2A$&$2BC$	&$2D$	&$3B$	&$3AC$	&$3D$	&--		&$7A$	&--		&$9AB$	&--\\
$\F$		&$2A$	&$2BC$	&$2D$	&$3A$	&$3B$	&$3C$	&--		&$7B$	&$7A$	&$9B$	&$9A$\\
\end{tabular}
\caption{Conjugacy class fusion} \label{fusion}
\end{table}

\begin{table}[h]\centering
\begin{tabular}{c c c c c c c c c c c c }
$L_6(2)$-Type	&$H$			&$N_{L}(H)$				&$\E$-type\\
\hline
$(2A,3B,7CD)$	&$L_3(2)$			&$L_3(2) \times L_3(2)$		&$(2A,3A,7C)$\\
\hline
$(2B,3B,7CD)$	&$L_3(2)$			&$L_3(2) \times S_4$		&$(2B,3A,7C)$\\
			&$L_3(2)$			&$L_3(2) \times S_4$\\
\hline
$(2B,3C,7AB)$	&$L_3(2)$			&$L_3(2) \times S_3$		&$(2B,3B,7AB)$\\
\hline
$(2B,3C,7E)$	&$L_3(2)$			&$L_3(2)$					&$(2B,3B,7D)$\\
			&$L_3(2)$			&$L_3(2) \sdp 2$\\
\hline
$(2C,3A,7E)$	&$L_2(8)$			&$(7 \times L_2(8)) \sdp 3$	&$(2C,3C,7D)$\\
\hline
$(2C,3C,7E)$	&$2^3 \ns L_3(2)$	&$2^3 \ns L_3(2)$			&$(2C,3B,7D)$\\
\hline
\end{tabular}
\caption{The Hurwitz subgroups of $L_6(2)$} \label{l62table}
\end{table}

\begin{table}[h]\centering
\begin{tabular}{c c c c c c c c c c c c }
$\OO$-Type	&$H$			&$N_{O}(H)$						&$\E$-type\\
\hline
$(2A,3B,7A)$	&$L_3(2)$			&$(L_3(2) \times S_3 \times S_3) \sdp 2$	&$(2A,3A,7C)$\\
\hline
$(2B,3B,7A)$	&$L_3(2)$			&$L_3(2) \times S_3 \times S_3$		&$(2B,3A,7C)$\\
\hline
$(2C,3B,7A)$	&$L_3(2)$			&$L_3(2) \times 2^2$				&$(2B,3A,7C)$\\
\hline
$(2D,3B,7A)$	&$L_3(2)$			&$L_3(2) \times 2^2$				&$(2C,3A,7C)$\\
			&$L_3(2)$			&$(L_3(2) \times 2^2) \sdp 2$\\
			&$2^3 \ns L_3(2)$	&$2^3 \ns L_3(2) \times 2^2$\\
			&$2^3 \ns L_3(2)$	&$2^3 \ns L_3(2) \times S_3$\\
\hline
$(2D,3D,7A)$	&$L_3(2)$			&$(3 \times L_3(2)) \sdp 2$			&$(2C,3C,7C)$\\
			&$L_2(8)$			&$L_2(8) \sdp 3$\\
			&$L_2(8)$			&$2 \times L_2(8)$\\
			&$L_2(8)$			&$S_3 \times L_2(8) \sdp 3$\\
\hline
\end{tabular}
\caption{The Hurwitz subgroups of $\OO$} \label{o10table}
\end{table}

\begin{table}[h]\centering
\begin{tabular}{c c c c c c c c c c c c }
$\F$-Type		&$H$			&$N_F(H)$					&$\E$-type		
\\
\hline
$(2A,3A,7B)$	&$L_3(2)$			&$(L_3(2) \times L_3(2)) \sdp 2$	&$(2A,3A,7C)$		
\\
$(2B,3B,7A)$	&$L_3(2)$			&$(L_3(2) \times L_3(2)) \sdp 2$	&$(2B,3B,7D)$\\
\hline
$(2C,3A,7B)$	&$L_3(2)$			&$L_3(2) \times S_4$			&$(2B,3A,7C)$		
\\
$(2C,3B,7A)$	&$L_3(2)$			&$L_3(2) \times S_4$			&$(2B,3B,7D)$\\
\hline
$(2D,3A,7B)$	&$L_3(2)$			&$(L_3(2) \times D_8) \sdp 2$		&$(2C,3A,7C)$		
\\
			&$2^3 \ns L_3(2)$	&$2^3 \ns L_3(2) \times D_8$\\
			&$2^3 \ns L_3(2)$	&$2^3 \ns L_3(2) \times S_4$\\
$(2D,3B,7A)$	&$L_3(2)$			&$(L_3(2) \times D_8) \sdp 2$		&$(2C,3B,7D)$\\
			&$2^3 \ns L_3(2)$	&$2^3 \ns L_3(2) \times D_8$\\
			&$2^3 \ns L_3(2)$	&$2^3 \ns L_3(2) \times S_4$\\
\hline
$(2D,3C,7A)$	&$L_3(2)$			&$L_3(2) \sdp 2$				&$(2C,3C,7D)$		
\\
			&$L_2(8)$			&$2 \times L_2(8)$\\
			&$L_2(8)$			&$2 \times L_2(8)$\\
			&$L_2(8)$			&$2 \times L_2(8) \sdp 3$\\
			&$L_2(8)$			&$S_3 \times L_2(8) \sdp 3$\\
			&$\DD$			&$\DD \sdp 3$\\
$(2D,3C,7B)$	&$L_3(2)$			&$L_3(2) \sdp 2$				&$(2C,3C,7C)$\\
			&$L_2(8)$			&$2 \times L_2(8)$\\
			&$L_2(8)$			&$2 \times L_2(8)$\\
			&$L_2(8)$			&$2 \times L_2(8) \sdp 3$\\
			&$L_2(8)$			&$S_3 \times L_2(8) \sdp 3$\\
			&$\DD$			&$\DD \sdp 3$\\
\hline
\end{tabular}
\caption{The Hurwitz subgroups of $\F$} \label{f42table}
\end{table}
\end{document}